\documentclass[12pt,a4paper,reqno]{amsart}
\usepackage{amsmath}
\usepackage{amsfonts}
\usepackage{amssymb}

\newtheorem{theorem}{Theorem}[section]
\newtheorem{corollary}{Corollary}[section]
\newtheorem{lemma}{Lemma}[section]

\def\F{{\mathbb F}}
\def\tF{\textsc{F}}

\def\O{{\mathcal O}}

\def\Q{{\mathbb Q}}

\def\V{\textsc{V}}

\def\Z{{\mathbb Z}}

\DeclareMathOperator{\coh}{H}
\DeclareMathOperator{\Frob}{Frob}
\DeclareMathOperator{\Gal}{Gal}

\DeclareMathOperator{\Ext}{Ext}
\DeclareMathOperator{\G_m}{G_m}

\DeclareMathOperator{\image}{Im}
\DeclareMathOperator{\Sel}{Sel}

\DeclareMathOperator{\Hom}{Hom}

\DeclareMathOperator{\coker}{coker}
\DeclareMathOperator{\spec}{spec}

\begin{document}
\title[Control Theorems over Global Function
Fields]{Control Theorems for
Abelian Varieties over Function
Fields}
\author{Ki-Seng Tan}
\address{Department of Mathematics\\
National Taiwan University\\%
Taipei 10764, Taiwan}
\email{tan@math.ntu.edu.tw}
\thanks{\textbf{Acknowledgement:} This research was supported in part by the National
Science Council of Taiwan, NSC95-2115-M-002-017-MY2.}
\begin{abstract}
We prove control theorems for abelian varieties over function fields.
\end{abstract}

\maketitle
\begin{section}{Introduction}
Control theorems for abelian varieties over $\Z^d$-extensions of global (or local) fields aim at
giving estimates for the sizes of their Galois (or flat) cohomology groups.
They play a crucial role in the arithmetic of abelian varieties and have important application
to Iwasawa theory. Control theorems are well known over number fields. In this paper
we prove control theorems for abelian varieties over function fields.

Our first result is the following. Let $p$ denote a given prime number and let $q$ be a power of $p$. We denote by $\F_q$ the finite field of order $q$.
\begin{theorem}{\em{(The Local Control Theorem)}}\label{t:loccont} Let $A$ be
an abelian variety over the local field
$K=\F_q((T))$. Assume that the reduction $\bar A$ of $A$ modulo $(T)$ is an
ordinary abelian variety. Write ${\bar A}(\F_q)_p$ for the $p$-Sylow subgroup of
${\bar A}(\F_q)$. Then for any $\Z_p^d$-extension $L/K$, we have the following estimate on
the size of the Galois cohomology group:
$$|\coh^1(L/K, A(L))|\leq |{\bar A}(\F_q)_p|^{d+1}.$$
\end{theorem}
Similar result over local fields of characteristic zero has been proved by Mazur (\cite{maz})
and Greenberg (\cite{gre}). These proofs depend on deep theorems on local points of abelian
varieties. However, some of these theorems may not hold in characteristic $p$.
For example, one remarkable phenomenon is that the $p$-completion of the rational
points of an abelian variety over a local field of characteristic $p$ is not a finitely
generated $\Z_p$-module. Our proof (in section \ref{s:insp}) introduces a new technique involving
a huge purely inseparable extension and hence works well in the characteristic $p$ situation.

Given an abelian variety $A$ over a field $K$ of characteristic $p$, we regard $A$ as a sheaf for the flat topology on $K$ and
denote the kernel of the multiplication by $p^m$ on $A$ by $\mathcal{A}[p^m]$. 
If $K$ is a global field, the $p^m$-Selmer group
$\Sel_{p^m}(F)$ for a finite extension field $F$ of $K$ is defined to be the kernel of the composition
$$\coh^1(F,\mathcal{A}[p^m])\longrightarrow \coh^1(F,A)\stackrel{loc}{\longrightarrow}
\bigoplus_v \coh^1(F_v,A),$$
where $loc$ is the localization map and in the direct sum, $v$ runs through all places of $F$.
The direct limit of $\Sel_{p^m}(F)$ as $m\rightarrow\infty$ is denoted by $\Sel_{p^{\infty}}(F)$.
For any Galois extension $L/K$, the $p$-primary part of the Selmer group of $A$ over $L$ is taken to be the direct limit of
$\Sel_{p^{\infty}}(F)$ over all finite intermediate fields $F$ of $L/K$.
We write $\Gamma_F$ for the Galois group of $L/F$ and let
$$res_{L/F}:\Sel_{p^{\infty}}(F)\longrightarrow \Sel_{p^{\infty}}(L)^{\Gamma_F}$$
be the restriction map. We deduce from Theorem \ref{t:loccont} the following control theorem.
\begin{theorem}\label{t:cont} {\em {(The Control Theorem)}} Let $L$ be a $\Z_p^d$-extension of a global field $K$ of characteristic $p$ with Galois
group $\Gal(L/K)=\Gamma$. Assume that $L/K$ is unramified outside a finite set $S$ of places of $K$. Let $A$ be an abelian variety over $K$ with good ordinary reduction
at every place in $S$. Then for every finite intermediate extension $F$ of $L/K$, the kernel and the cokernel of the restriction map
$res_{L/F}$ on the $p$-primary Selmer groups $\Sel_{p^{\infty}}(F)$ are finite. Furthermore, if $d=1$, then the orders of the kernel and the cokernel of $res_{L/F}$
are bounded as $F$ varies.
\end{theorem}
The number field counterpart of this theorem appears in
Mazur (\cite{maz}) and Greenberg (\cite{gre}).

To show that our control theorem is very useful we give an application to Iwasawa theory. Denote by $\Lambda_{\Gamma}$ the Iwaswa algebra $\Z_p[[\Gamma]]$
and denote the Pontryagin dual $\Hom(\Sel_{p^{\infty}}(L),\Q_p/\Z_p))$ by $X_L$.
\begin{theorem}\label{t:iwasawa}
If the condition of {\em {Theorem \ref{t:cont}}} holds, then $X_L$ is a finitely generated module over $\Lambda_{\Gamma}$.
\end{theorem}

This theorem is fundamental for advances in Iwasawa theory. Such results
in the function field case only appear recently.
Lacking our local control theorem, these results have to depend on extra assumptions.
For example, in Ochiai and Trihan (\cite{otr06, otr08}), they assume that $L/K$ is the constant $\Z_p$-extension unramified at every place of $K$, while
Bandini and Longhi (\cite{bl06}) treat the case of elliptic curve with split multiplicative reduction at every place of $S$.
On the other hand, our control theorem reduces the proof of Theorem \ref{t:iwasawa} to a routine task, because it obviously implies that
the $p$-torsion subgroup $\Sel_{p^{\infty}}(L)^{\Gamma}[p]$ of
$\Sel_{p^{\infty}}(L)^{\Gamma}$ is a finite group and hence the Nakayama Lemma (see \cite{w82}, p.279) can be applied to the compact $\Lambda_{\Gamma}$-module $X_L$.
Here is another
application of our main results, which strengthen 
Theorem 1.1 of \cite{bl06}.
\begin{theorem}\label{t:bl06}
Suppose that $A/K$ is an elliptic curve and at every place $v\in S$, $A$ has either good ordinary or split multiplicative reduction.
Then $X_L$ is a finitely generated $\Lambda_{\Gamma}$-module.
\end{theorem}
In Section \ref{subsec:cont} we prove our main result Theorem
\ref{t:cont} as well as Theorem \ref{t:bl06}, assuming Theorem \ref{t:loccont}
whose proof we postpone to Section \ref{subsec:t1}.

Finally, we set some notations. From now on, every field will be of characteristic $p>0$.
For an abelian variety $A$ defined over a field $K$,
denote by $A[p^m]$ the $p^m$-torsion points on $A$ and write
$A[p^{\infty}]=\bigcup_m A[p^m]$. Denote $A(K)_{tor,p}=A[p^{\infty}]\cap A(K)$, the
$p$-primary part of $A(K)_{tor}$.

Suppose that $K$ is a local field. We use $\F_K$ to denote its residue field which can also be viewed as its constant field.
Let $A^1(K)$ denote the pro-$p$ subgroup of $A(K)$ consisting of points with trivial reduction, and let $A(K)_p$ denote the maximal pro-$p$ subgroup
of $A(K)$.

For a global or local field $K$ and for each $n$,
we use $K^{(1/p^n)}/K$ to denote
the unique purely inseparable extension of degree $p^n$.
Also, we use $\bar K$ to
denote the separable closure of $K$ and write $G_K=\Gal({\bar K}/K)$.
We write
${\bar K}^{(1/p^n)}$ for ${\overline{K^{(1/p^n)}}}$.
Thus, the algebraic closure
of $K$ equals ${\bar K}^{(1/p^{\infty})}:=\bigcup_{n=1}^{\infty} {\bar K}^{(1/p^n)}$.
The Frobenius substitution
$$\Frob_{p^n}:K^{(1/p^n)}\longrightarrow K, \;\; x\mapsto x^{p^n},$$
is an isomorphism. And we use it to identify $G_{K^{(1/p^n)}}$, for $n=1,...,\infty$,
with $G_K$.

The author would like to thank A. Bandini, W.-C. Chi, C. D. Gonz$\acute{\text{a}}$lez-Avil$\acute{\text{e}}$s, K.F. Lai,
D. Rockmore and F. Trihan for many valuable suggestions.
\end{section}

\begin{section}{The $p^m$-torsion points}\label{s:sel}
In this section, we prove that Theorem \ref{t:loccont}
implies Theorem \ref{t:cont} and Theorem \ref{t:bl06}.


\begin{subsection}{Ordinary abelian varieties}\label{subsec:ordinary}
Assume that $K$ is a field of characteristic $p$ and $A/K$ is an
abelian variety of dimension $g$. Then $A/K$ is ordinary
if and only if (over the algebraic closure of $K$)
the group scheme ${\mathcal{A}}[p]$
can be decomposed as (cf. \cite{mum74}, Sec. 14):
\begin{equation}\label{e:decom}
{\mathcal{A}}[p]=(\Z/p\Z)^g\times (\mu_p)^g.
\end{equation}
This condition is equivalent to
\begin{equation}\label{e:ppt}
A[p^m]\simeq (\Z/p^m\Z)^g.
\end{equation}
In this case, the multiplication by $p$ on $A$ is decomposed as
\begin{equation}\label{e:p}
[p]=\V\circ \tF,
\end{equation}
where $\tF:A \longrightarrow A^{(p)}$
is the Frobenius isogeny and $\V:A^{(p)}\longrightarrow A$ is separable.
\begin{lemma}\label{l:lift}
Let $A$ be an abelian variety over a local field $K$ so that $\bar A$, the reduction of $A$, is an ordinary
abelian variety. Then $A$ is also ordinary and the reduction map induces an isomorphism of $G_K$-modules:
\begin{equation}\label{e:ord}
A[p^{\infty}]\stackrel{\sim}{\longrightarrow} {\bar A}[p^{\infty}]
\end{equation}
\end{lemma}
Here $G_K$ acts on $A[p^{\infty}]$ through its action on ${\bar K}^{(1/p^{\infty})}$ and its action on ${\bar A}[p^{\infty}]$
factors through the quotient map $G_K\longrightarrow \Gal({\bar \F}_K/\F_K)$.
\begin{proof}
Let $m$ be any given positive integer.
By replacing $K$ with a suitable finite (maybe inseparable) extension field of it, we may assume that $A[p^m]$ is rational
over $K$ and ${\bar A}[p^m]$ is rational over $\F_K$.
Let $\O$ be the ring of integers of $K$ and denote by $\mathbf{A}$ the N$\acute{\text{e}}$ron model of $A$ over $\O$.
Since $\mathbf{A}[p^m]:=\ker (\mathbf{A}\stackrel{p^m}{\longrightarrow}\mathbf{A})$ is proper, quasi-finite and hence finite over $\O$ and
$\O$ is a complete discrete valuation ring, we have $\mathbf{A}[p^m]=\spec B$ with $B=\prod B_i$ where each $B_i$ is a local ring over $\O$
(\cite{mil80} I.4.2(b)). Since ${\bar A}[p^m]$ is rational over $\F_K$, the residue field of each $B_i$ equals $\F_K$ and hence
the reduction map from $\Hom_{\O} (B,\O)=\Hom_{\O} (B,K)=A[p^m]$ to $\Hom_{\O}(B,\F_K)={\bar A}[p^m]$
is surjective. The lemma is proved, since ${\bar A}[p^m]\simeq (\Z/p^m\Z)^g$ and the order of $A[p^m]$ is at most $p^{gm}$.
\end{proof}
\begin{corollary}\label{c:list} Let $A$ be an abelian variety over a local field $K$ with good ordinary reduction and let $L$ be a local field
containing $K$. Then the following hold:
\begin{description}
\item[(a)] If $P$ is a point in $A(L)$, then
all the $p^m$-division points of $P$ are contained in $A({\bar L}^{(1/p^m)})$.
In particular, the $p^m$-torsion points $A[p^m]\subset A({\bar K}^{(1/p^m)})$.
\item[(b)] If $L/K$ is separable, then the group $A(L)_{tor,p}$ is unramified over $K$,
in the sense that every point in $A(L)_{tor,p}$ is rational over the
maximal unramified sub-extension of $L/K$.
\item[(c)] The subgroup $A^1(L)\subset A(L)$ is a torsion free $\Z_p$-module.
\item[(d)] For each $P\in A^1(L)$ there is a unique $P'\in A^1(L^{(1/p^m)})$ such that
$p^mP'=P$, and vice versa. In other words, we have
\begin{equation}\label{e:pm}
A^1(L)= p^mA^1(L^{(1/p^m)}).
\end{equation}
\item[(e)] Let $A(L^{(1/p^{\infty})})_p:=\bigcup_m A(L^{(1/p^m)})_p$. Then
\begin{equation}\label{e:dp}
A(L^{(1/p^{\infty})})_p=A^1(L^{(1/p^{\infty})})
\times A(L^{(1/p^{\infty})})_{tor,p},
\end{equation}
and
\begin{equation}\label{e:gk}
A(L^{(1/p^{\infty})})_{tor,p}\simeq {\bar A}(\F_L)_p.
\end{equation}
If $L/K$ is Galois, then these are $G_K$-isomorphisms.
\end{description}
\end{corollary}

\begin{proof}
The statement (a) is directly from the arguments given before the lemma, while
(b) and (c) are from the $G_K$-isomorphism (\ref{e:ord}).

To see (d), let $Q\in A({\bar L}^{(1/p^m))}$ be a $p^m$-division
point of $P\in A^1(L)$. Since the reduction ${\bar Q}$ is contained in
${\bar A}[p^m]$, there is a point
$R\in A[p^m]\subset A({\bar L}^{(1/p^m)})$ such that
$P':=Q-R\in A^1({\bar L}^{(1/p^m)})$.  
Obviously,  $P'$ is also a $p^m$-division
point of $P$, and for $\sigma\in G_L$, we have (from (c))
$$^{\sigma}P'-P'\in A[p^m]\cap A^1(
{\bar L}^{(1/p^m)})=\{0\}.$$

To prove (e), for each $m$ we consider
the following commutative diagram of exact sequences induced from reduction
maps:
$$
\begin{array}{ccccccc}
0\longrightarrow & A^1(L) & \longrightarrow & A(L)_p & \longrightarrow & {\bar A}(\F_L)_p
& \longrightarrow 0\\
{} & \downarrow & {} & \downarrow & {} & \parallel & {} \\
0\longrightarrow & A^1(L^{(1/p^m)}) & \longrightarrow & A(L^{(1/p^m)})_p & \longrightarrow &
{\bar A}(\F_L)_p
& \longrightarrow 0,\\
\end{array}
$$
where the down-arrows are natural inclusions. Suppose $Q\in A(L)_p$ so that its reduction
${\bar Q}\in {\bar A}(\F_L)_p$ has order $p^n$, and let $P=p^nQ\in A^1(L)$.
According to (d), there is a $P'\in A^1(L^{(1/p^n)})$ such that $p^nP'=P$. Then the point
$Q-P'$ is contained in $A(L^{(1/p^n)})_{tor,p}$
and the assignment ${\bar Q}\mapsto Q-P'$ define a natural isomorphism
from ${\bar A}(\F_L)_p$ to $A(L^{(1/p^m)})_{tor,p}$ for large $m$.
Thus, if $m$ is large enough, then there is a natural splitting of the exact sequence
$$
0\longrightarrow  A^1(L^{(1/p^m)})  \longrightarrow  A(L^{(1/p^m)})_p  \longrightarrow
{\bar A}(\F_L)_p
 \longrightarrow 0.
$$
Thus, we obtain (\ref{e:dp}) and (\ref{e:gk}) by letting $m\rightarrow\infty$.
\end{proof}
Next, we consider the case where the abelian variety $A$ is defined over a global field $K$ of characteristic $p$.
Let $L/K$ be a $\Z_p^d$-extension
unramified outside a finite set $S$ of places of $K$.
Let $\Gamma_0$ denote the stabilizer of $A(L)_{tor,p}$ for the action of $\Gamma:=\Gal(L/K)$
and let $L_0$ denote the fixed field of $\Gamma_0$.
We call a pro-$p$ Galois extension pro-$p$ cyclic if its Galois group is
either finite cyclic or isomorphic to $\Z_p$.

\begin{lemma}\label{l:globordin}
Let notations be as above.
Assume that one of the following holds:
\begin{enumerate}
\item $A$ has good, ordinary reduction at every place $v\in S$.
\item $A$ is an elliptic curve
with good, ordinary or split multiplicative reduction at every place of $S$.
\end{enumerate}
Then there is a finite intermediate extension
$K_0/K\subset L_0/K$ such that $L_0/K_0$
is a pro-$p$ cyclic extension.
\end{lemma}
\begin{proof}
Note that if $A$ has good, ordinary reduction at a place $v$, then
$A/K_v$ is ordinary and hence $K_vL_0$ is unramified over $K_v$ (Corollary \ref{c:list}(b)).
Also, if $A$ is an elliptic curve with split multiplicative reduction at some $v\in S$ so that the Tate's curve has period $Q\in K_v^*$,
then $K_vL_0={\overline {K_v}}\cap ( \cup_{n=1}^{\infty} K_v(Q^{1/p^n}))=K_v$. This shows that $L_0/K$ is everywhere unramified.

We then apply the global class field theory (c.f. \cite{t67})
which tells us that the Galois group $W_{K,p}$
of the maximal everywhere unramified pro-$p$ abelian extension of $K$ fits into an exact sequence
$$0\longrightarrow C_{K,p} \longrightarrow W_{K,p}\stackrel{\deg}{\longrightarrow}
\Z_p\longrightarrow 0,$$
where $C_{K,p}$ is the $p$-Sylow subgroup of the class group of $K$ and $\deg$ is induced
from the degree map on the group of ideles. We choose a subgroup $W_0\simeq\Z_p$ of $W_{K,p}$
and choose $K_0$ to be the fixed field of $W_0$ under the action of $W_{K,p}$ on $L_0$.
\end{proof}

\end{subsection}

\begin{subsection}{Cohomology groups of the torsion points}\label{subsec:torsion}
In the next step, our goal is to bound, for $i=1,2$, the order of the cohomology group
$\coh^i(L'/K, A(L')_p)$, where $L'/K$ is a finite intermediate field
extension of some given $\Z_p^d$-extension. To achieve this goal, we first
establish the following lemma in which $G$ is
a finite $p$-abelian group with $d$ generators acting on a finite $p$-abelian group 
$M$. We assume that there is a subgroup $H_0\subset G$
such that $G/H_0$ is cyclic and $M^{H_0}=M$.

\begin{lemma}\label{l:incor} Let notations and conditions be as above.
Then we have
\begin{equation}\label{e:in1}
|\coh^1(G, M)|\leq |M^G|^{d},
\end{equation}
and
\begin{equation}\label{e:in2}
 |\coh^2(G, M)|\leq |M^G|^{d^2}.
\end{equation}
\end{lemma}
\begin{proof} Consider the inflation-restriction exact sequence:
$$0\longrightarrow \coh^1(G/H_0, M^{H_0})\longrightarrow
\coh^1(G, M)\stackrel{res}{\longrightarrow}
\coh^1(H_0,  M)^{G/H_0}.$$
We shall bound the orders of $\ker(res)$ and $\image(res)$.
Since $G/H_0$ is cyclic, by computing the Herbrand quotient, we see that the order of
$\coh^1(G/H_0, M^{H_0})$ equals $|M^{G}/\mathcal{N}|$,
where $\mathcal{N}$ is the image of the norm map $N_{G/H_0}:M=M^{H_0}\longrightarrow
M^{G}$. Also, since $M$
is fixed by the action of $H_0$, we have
$$\coh^1(H_0,  M)^{G/H_0}=\Hom (H_0, M^{G}).$$
To proceed further, choose a basis $e_1,...,e_c$ of $G$, for some $c\leq d$, so that
$e_1':=p^me_1,e_2,...,e_c$, for some non-negative integer $m$, form a basis of $H_0$.
The cocycle condition implies that if $\rho$ is a $1$-cocycle representing a class in
$\coh^1(G, M)$, then the value $\rho(e_1')$ equals $N_{G/H_0}(\rho(e_1))$.
This implies that
the image of $res$ must be contained in the subgroup
$$\{ \phi\in \Hom (H_0, M^{G})\; | \; \phi(e_1')\in\mathcal{N}\},$$
whose order is bounded by $|M^{G}|^{c-1}\cdot |\mathcal{N}|$.
Therefore, the inequality (\ref{e:in1}) holds, since
$$| \ker(res)|\cdot |\image(res)|\leq
|M^{G}/\mathcal{N}|\cdot |M^{G}|^{c-1}\cdot |\mathcal{N}|.
$$

We prove the inequality (\ref{e:in2}) by induction on $d$.
The case where $d=1$ is easy, since
$\coh^2(G,M)=M^G/N_{G}(M)$.
If $d>1$,  we choose a cyclic subgroup $H_1\subset H_0$ such that
$G/H_1$ is generated by $d-1$ elements.
According to
the associated Hochschild-Serre
spectral sequence (cf. \cite{sh72}),
we have exact sequences
$$0\longrightarrow E_1^2\longrightarrow \coh^2(G, M)\longrightarrow
\coh^2(H_1, M)^{G/H_1}
$$
and
$$\coh^2(G/H_1, M^{H_1})\longrightarrow E_1^2
\longrightarrow \coh^1(G/H_1, \coh^1(H_1,M)).
$$
Therefore,
the desired bound for the order of $\coh^2(G,M)$
can be derived from the following lemma.
\end{proof}

\begin{lemma}\label{l:ll}
Under the above assumptions, we have
\begin{equation}\label{e:1}
|\coh^2(G/H_1, M^{H_1})|\leq |M^G|^{(d-1)^2},
\end{equation}
\begin{equation}\label{e:2}
|\coh^1(G/H_1, \coh^1(
H_1, M))|\leq |M^G|^{d-1},
\end{equation}
\begin{equation}\label{e:3}
|\coh^2(
H_1, M)^{G/H_1}|\leq |M^G|^{d}.
\end{equation}
\end{lemma}
\begin{proof}
The inequality (\ref{e:1}) is in fact the induction hypothesis.
To show (\ref{e:2}), we first note that since $H_1$ is cyclic and acting trivially
on $M$, the group $N:=\coh^1(H_1, M)$
satisfies $N^G=\Hom (H_1, M^G)$ and $|N^G|\leq |M^G|$.
In view of this, we see that the inequality
(\ref{e:in1}) for the case where the pair $(G,M)=(G/H_1,N)$ implies (\ref{e:2}).

Again, since $H_1$ is cyclic, acting trivially on $M$, we have
$$\coh^2(H_1, M)^{G/H_1}=(M/p^lM)^{G/H_1},$$
where
$p^l$ is the order of $H_1$. To bound the order of this group,
we consider the exact sequence
$$M^G \longrightarrow (M/p^lM)^{G/H_1} \longrightarrow \coh^1(G/H_1, p^lM^G),$$
which is induced from
$$0\longrightarrow p^lM\longrightarrow M
\longrightarrow
M/p^lM \longrightarrow 0.$$
We have
$$|\coh^1(G/H_1,p^lM^G)|=|\Hom (G/H_1,p^lM^G)|\leq |M^G|^{d-1}.$$

\end{proof}

\begin{corollary}\label{c:loc}
Suppose that $K$ is a local field of characteristic $p$,
$L/K$ is a $\Z_p^d$-extension and $A/K$ is an abelian variety with good, ordinary reduction.
Then for every finite intermediate extension $L'/K\subset L/K$ we have
$$
|\coh^1(L'/K, A(L')_{tor,p})|\leq |A(K)_{tor,p}|^d,
$$
and
$$
|\coh^2(L'/K, A(L')_{tor,p})|\leq |A(K)_{tor,p}|^{d^2}
$$
\end{corollary}
\begin{proof}
Corollary \ref{c:list}(b)) says that $A(L)_{tor,p}$ is unramified.
Let $L_0/K$ be the maximal unramified subextension of $L/K$ and put
$G=\Gal(L'/K)$, $H_0=\Gal(L'/L_0\cap L')$.
Then apply Lemma \ref{l:incor}.
\end{proof}

\begin{corollary}\label{c:glob}
Suppose that $A$, $K$, $L$ satisfy the condition of {\em{ Lemma \ref{l:globordin}}}.
Let $F/K$ be a finite intermediate extension of $L/K$.
Then for all intermediate extension $L'/F\subset L/F$, the orders of
$\coh^1(L'/F, A(L')_{tor,p})$
and $\coh^2(L'/F, A(L')_{tor,p})$ are bounded. Furthermore, if $d=1$, then the bounds can be chosen to be independent of $F$.
\end{corollary}
\begin{proof}
Let $K_0\subset L_0\subset L$ be as in Lemma \ref{l:globordin}.
Without loss of generality, we may assume that $F=K$ for the proof of the first statement.
Put $K_0'=L'\cap K_0$, $G=\Gal(L'/K_0')$ and $H_0=\Gal(L'/L_0\cap L')$.
Obviously, $A(K_0')_{tor,p}$ is contained in $A(K_0)_{tor,p}$. Therefore, from
Lemma \ref{l:incor} we see that for $j=0,1,2$ the order of the $\Gal(K_0'/K)$-module
$\coh^j(L'/K_0', A(L')_{tor,p})$ is bounded by $|A(K_0)_{tor,p}|^{d^j}$ which is independent of $L'$.
This implies that the orders $|\coh^i(K_0'/K,\coh^j(L'/K_0', A(L')_{tor,p}))|$, for $i+j=1,2$,
are also bounded. Then we use the
Hochschild-Serre spectral sequence
$$\coh^i(K_0'/K,\coh^j(L'/K_0', A(L')_{tor,p})) \Longrightarrow \coh^{i+j}(L'/K,A(L')_{tor,p})$$
to verify the first statement.

Now consider the case where $d=1$. Let $K_n$ be the $n$th layer of $L/K$.
Using Herbrand quotient, we see that for $F=K_n$,
$$|\coh^1(L'/F, A(L')_{tor,p})|=|\coh^2(L'/F, A(L')_{tor,p})|\leq |A(K_n)_{tor,p}|.$$
This bound increases with $n$. To find a bound independent of $n$, we first note that 
$A(L)_{tor,p}$ is cofinite over $\Z_p$ and
consider the $p$-divisible part
$A(L)_{tor,\infty}$ of $A(L)_{tor,p}$.  Let $T$ denote the finite quotient
$A(L)_{tor,p}/A(L)_{tor,\infty}$, and let $n_0$ be a positive integer
such that if $n\geq n_0$, then $A(K_n)_{tor,p}$ contains $A[p^2]\cap A(L)_{tor,p}$.

Suppose $n\geq n_0$ and $Q\in A(K_n)\cap A(L)_{tor,\infty}$.
It is easy to see that there is a $Q'\in  A(K_{n+1})\cap A(L)_{tor,\infty}$ such that $pQ'=Q$ and for every
$\sigma\in\Gal(K_{n+1}/K_n)$,
the point $P_{\sigma}:={}^{\sigma}Q'-Q'$ is contained in $A[p]\cap A(L)_{tor,\infty}$.
Note that $A[p]\cap A(L)_{tor,\infty}$ is a subgroup of $p(A(K_n)\cap A(L)_{tor,\infty})$ which is contained in ${\text{N}}_{K_{n+1}/K_n}(A(K_{n+1})\cap A(L)_{tor,\infty})$.
Also, $Q$ can be expressed as the difference
${\text{N}}_{K_{n+1}/K_n}(Q')-\sum_{\sigma\in\Gal(K_{n+1}/K_n)} P_{\sigma}$. Therefore, $Q$
is contained in ${\text{N}}_{K_{n+1}/K_n}(A(K_{n+1})\cap A(L)_{tor,\infty})$.
This shows that
$$A(K_n)\cap A(L)_{tor,\infty}\subset {\text{N}}_{K_{m}/K_n}(A(K_{m}\cap A(L)_{tor,\infty})),\;\;\text{if}\;\; m\geq n.$$
Therefore, we have, for $F=K_n$, $L'=K_m$, $m\geq n\geq n_0$,
$$|\coh^2(L'/F,A(L')_{tor,p})|=|A(K_n)_{tor,p}/{\text{N}}_{K_{m}/K_n}(A(K_{m})_{tor,p})|\leq |T|.$$
We can choose $|T|$ as the desired bound.
\end{proof}

\end{subsection}

\begin{subsection}{The proofs of Theorem \ref{t:cont} and Theorem \ref{t:bl06}}\label{subsec:cont}
We first prove that Theorem \ref{t:loccont} implies Theorem \ref{t:cont}.
Let $S(F)$ denote the set of places of $F$ sitting over $S$.
Let $L'/F$ be a finite intermediate extension of $L/F$ and put $G=\Gal(L'/F)$. For $m=1,2,...,\infty$, consider the restriction map
$$res_m:\coh^1(F,\mathcal{A}[p^m])
\longrightarrow \coh^1(L',\mathcal{A}[p^m])^G,$$
and define
$$ \Sel_{p^{\infty}}(L'/F):=
\{ \eta \in \coh^1(F,\mathcal{A}[p^{\infty}])\; | \; res_{\infty} (\eta)\in \Sel_{p^{\infty}}(L')\}.$$
Then $\Sel_{p^{\infty}}(F)\subset \Sel_{p^{\infty}}(L'/F)$ and for the restriction map
$$res_{L'/F}:\Sel_{p^{\infty}}(F)\longrightarrow \Sel_{p^{\infty}}(L')^G,$$
we have the inequalities:
\begin{equation}\label{e:kerbd}
|\ker(res_{L'/F})|\leq |\ker(res_{\infty})|,
\end{equation}
and
\begin{equation}\label{e:cokbd}
|\coker(res_{L'/F})|\leq |\coker(res_{\infty})|\cdot |\Sel_{p^{\infty}}(L'/F):\Sel_{p^{\infty}}(F)|.
\end{equation}
For every $m$ apply the Hochschild-Serre
spectral sequence (\cite{mil80}, p. 105)
$$ \coh^i(G, \coh^j(L',\mathcal{A}[p^m]))\Longrightarrow \coh^{i+j}(F, \mathcal{A}[p^m]).$$
The spectral sequence says that $\ker(res_m)$
equals $\coh^1(G, \mathcal{A}[p^m](L'))$ and $\coker(res_m)$ is isomorphic to a subgroup of
$\coh^2(G, \mathcal{A}[p^m](L'))$.
We have $\mathcal{A}[p^m](L')=A(L')[p^m]$, which equals $A(L')_{tor,p}$ for $m$ large enough.
By letting $m$ go to
$\infty$ and by applying Corollary \ref{c:glob}, we conclude that the orders $|\ker(res_{\infty})|$,
$|\coker(res_{\infty})|$ are finite and they are bounded if $d=1$ and $F$ varies.

To bound the index $|\Sel_{p^{\infty}}(L'/F):\Sel_{p^{\infty}}(F)|$, we use the exact sequence
\begin{equation}\label{e:slf}
\Sel_{p^{\infty}}(F)\longrightarrow \Sel_{p^{\infty}}(L'/F)\longrightarrow \bigoplus_v\coh^1(L_v/F_v,A(L_v)),
\end{equation}
where in the right term $v$ runs through all places of $F$.
For each $v$, let $\F_v$ be the residue field and let $m_v$ be the number of
components of the special fiber of the N$\acute{\text{e}}$ron model of $A$ at $v$.
We first note that if $v$ split completely in $L$, then the cohomology group $\coh^1(L_/K_v,A(L_v))=0$.
Then we apply
Theorem \ref{t:loccont} (for $v\in S$) together
with Proposition I.3.8 of \cite{mil86} (for $v\notin S$ and hence unramified) to show
that the index is bounded by the product
$\mathbf{B}_F:=\prod_{v\in S(F)}|{\bar A}(\F_v)_p|^{d+1}\cdot\prod_{v\notin S(F)}m_v$,
where in the second product $v$ runs through all
the places not splitting completely in $L$.
We also note that this is a finite product, since $m_v=1$ if $A$ has good reduction at $v$.
Therefore, the index is finite and the first statement of Theorem \ref{t:cont} is proved.
Moreover, if $d=1$ and $v_0$ is a place of $K$ not splitting completely in $L$, then
the decomposition group of $v_0$ is a non-trivial closed subgroup of $\Gamma\simeq\Z_p$ with finite index,
and hence the number of place of $L$ sitting over $v_0$ is finite.
This implies that the number of place of $F$ sitting over $v_0$ is bounded as $F$ varies.
Therefore, the product $\mathbf{B}_F$ is bounded, and this completes the proof of Theorem \ref{t:cont}.

To prove Theorem \ref{t:bl06} we use Nakayama lemma. We need to show that for each finite intermediate extension $L'/K\subset L/K$, the
order of the $p$-torsion subgroup of $\coker(res_{L'/K})$ is bounded. We apply Corollary \ref{c:glob} and use an argument similar to
the above. Then we reduce the proof to showing that for each $v\in S$ at which $A$ has split multiplicative reduction, the $p$-torsion subgroup of
$\coh^1(L'_v/K_v,A(L'_v))$ is bounded as $L'$ varies. For this, we use the exact sequence
$$0\longrightarrow Q^{\Z}\longrightarrow (L'_v)^*\longrightarrow A(L'_v)\longrightarrow 0,$$
where $Q$ is the local Tate's period. Hilbert's theorem 90 implies that $\coh^1(L'_v/K_v,A(L'_v))$ is isomorphic to a subgroup of
$$\coh^2(\Gal(L'_v/K_v),Q^{\Z})\simeq \coh^2(\Gal(L'_v/K_v),\Z)\simeq \coh^1(\Gal(L'_v/K_v),\Q/\Z).$$
Obviously, the order of its $p$-torsion subgroup is bounded by $p^d$.
\end{subsection}

\end{section}

\begin{section}{Inseparable Extensions}\label{s:insp}
In this section, we assume that $A$ is an abelian variety over the local field
$K=\F_q((T))$
so that the reduction $\bar A$ of $A$ is an ordinary abelian variety.

\begin{subsection}{The dual abelian variety}\label{subsec:dual}
Consider a finite extension $L/K$.
The Frobenius substitution $\Frob_{p}$ induces a $G_L$-isomorphism:
\begin{equation}\label{e:fr}
\begin{array}{rcl}
\Frob_{p}:A({\bar L}^{(1/p)}) & \stackrel{\sim}{\longrightarrow} & A^{(p)}(\bar L)\\
P & \mapsto & \tF(P).\\
\end{array}
\end{equation}
We apply (\ref{e:fr}) to the equality (\ref{e:pm}) which, according to
Corollary \ref{c:list}(c), can be written as
$A^1(L^{(1/p)})=1/p A^1(L)$. We then use the
relation (\ref{e:p}) to deduce a new equality:
\begin{equation}\label{e:surj}
\V((A^{(p)})^1(L))=\V(\tF(A^1(L^{(1/p)})))=\V(\tF(1/pA^1(L)))=A^1(L).
\end{equation}

Let $B$ denote the dual abelian variety
to $A$ over $K$. Being isogenous to $A$,
$B$ also has ordinary reduction. Let
${\hat \tF}: B^{(p)}\longrightarrow B$ be the dual
to the Frobenius isogeny $\tF$.
Then the kernel of $\hat {\tF}$, which is the dual of
$(\mu_p)^g$, is exactly the maximal etale subgroup of the group scheme $\mathcal{B}^{(p)}[p]$
(the kernel of the multiplication
by $p$ on $B^{(p)}$).
On the other hand, if we write $[p]_B$, the multiplication by $p$ on $B$, as the composition
$\V_B\circ \tF_B$, then $\V_B$ is separable and hence its kernel also
equals the maximal etale subgroup of $\mathcal{B}^{(p)}[p]$. In view of these, we see that
$\hat \tF=\V_B\circ \Phi$, for some isomorphism $\Phi:B^{(p)}\longrightarrow B^{(p)}$.
In particular, we have ${\hat \tF}((B^{(p)})^1(L))=\V_B((B^{(p)})^1(L))$.

By letting $B$ play the role of $A$ in (\ref{e:surj}), we prove the following.

\begin{lemma}\label{e:surjb} The map
$${\hat \tF}\mid_{(B^{(p)})^1(L)}:(B^{(p)})^1(L)\longrightarrow B^1(L)$$
is surjective.
\end{lemma}

\end{subsection}
\begin{subsection}{The local duality}\label{subsec:loc}
Via the Poincar$\acute{\text{e}}$ biextension $W\longrightarrow A\times B$ (which is the compliment of the zero section in the
Poincar$\acute{\text{e}}$ line bundle over
$A\times B$, \cite{mum68}),
a point on $B$ is regarded as an element in $\Ext(A,\G_m)$, and hence a point
$Q\in B(L)$ gives rise to an exact sequence of $G_L$-modules:
$$0\longrightarrow {\bar L}^*\longrightarrow W_Q\longrightarrow A({\bar L})\longrightarrow 0.$$
Using the induced long exact sequence:
$$\dots\longrightarrow  \coh^1(L,A){\stackrel{\delta_Q}{\longrightarrow}}
\coh^2(L,{\bar L}^*)\longrightarrow \dots,$$
we define (cf. \cite{mil86}, Appendix C)
the local pairing of $Q$ and a class $\xi\in\coh^1(L,A)$ as
$$<\xi,Q>_{A,B,L}:=inv(\delta_Q(\xi)).$$
Here $inv:\coh^2(L,{\bar L}^*)\longrightarrow
\Q/\Z$ is the invariant of the Brauer group.

If $W_A$, $W_{A'}$ are Poincar$\acute{\text{e}}$ biextensions associated to $A$, $A'$ and
$f:A\longrightarrow A'$ and ${\hat f}:B'\longrightarrow B$ are dual isogenies,
then $(1\times {\hat f})^* W_A\simeq (f\times 1)^* W_{A'}$ (\cite{mum74}, p.130).
From this, we see that the local pairings are compatible with isogenies.
In particular, we have the commutative diagram:
\begin{equation}\label{e:dual}
\begin{array}{lccc}
<,>_{A,B,L}: & \coh^1(L,A)\times B(L) & \longrightarrow & \Q/\Z\\
{} & \tF\downarrow\;\;\;\;\;\;\;\;\; \uparrow {\hat \tF}      &      {}        & \|\\
<, >_{A^{(p)},B^{(p)},L}: & \coh^1(L,A^{(p)})\times B^{(p)}(L) & \longrightarrow & \Q/\Z.
\end{array}
\end{equation}


We use (\ref{e:dual}) to deduce the following (dual
statement of Lemma \ref{e:surjb}).
\begin{lemma}\label{l:inj}
Let $\tF_*:\coh^1(L,A)\longrightarrow
\coh^1(L,A^{(p)})$ be the homomorphism
induced from $\tF$. If $\xi\in\coh^1(L,A)$ and $\tF_*(\xi)$ annihilates
$(B^{(p)})^1(L)$, then $\xi$ must annihilate $B^1(L)$.
\end{lemma}
We also have the commutative diagram induced from the Frobenius substitution:
$$
\begin{array}{lccc}
<,>_{A,B,L^{(1/p)}}: & \coh^1(L^{(1/p)},A)\times B(L^{(1/p)}) & \longrightarrow & \Q/\Z\\
{} & \Frob_{p*}\downarrow\;\;\;\;\;\;\;\;\; \downarrow \Frob_p      &      {}        & \|\\
<, >_{A^{(p)},B^{(p)},L}: & \coh^1(L,A^{(p)})\times B^{(p)}(L) & \longrightarrow & \Q/\Z.
\end{array}
$$

\begin{corollary}\label{c:inj1}
Let $\tF_{**}:\coh^1(L,A)\longrightarrow \coh^1(L^{(1/p)},A)$
be the homomorphism
induced from the inclusion $A({\bar L})\hookrightarrow
A({\bar L}^{(1/p)})$. If $\xi\in\coh^1(L,A)$ and $\tF_{**}(\xi)$ annihilates
$B^1(L^{(1/p)})$, then $\xi$ must annihilate $B^1(L)$.
\end{corollary}
\begin{proof}
We observe that $\tF_{**}$ is just the composition $\Frob_{p*}^{-1}\circ \tF_*$.
\end{proof}
Using Tate's local duality theorem (\cite{t62}, \cite{mil70}) which says that the local pairing is non-degenerate
and it identifies $\coh^1(L,A)$ with the Pontryagin dual of $B(L)$,
we identify the
annihilators of $B^1(K)$ with the dual group ${\widehat {\bar B}(\F_q)}$ of ${\bar B}(\F_q)={\bar B}(\F_K)$.
\begin{corollary}\label{c:inj2}
The kernel of the homomorphism
$$i_*:\coh^1(K,A)=\coh^1(G_K,A({\bar K}))\longrightarrow
\coh^1(G_K,A({\bar K}^{(1/p^{\infty})}))$$
induced from the inclusion
$A({\bar K})\longrightarrow A({\bar K}^{(1/p^{\infty})})$
is contained in ${\widehat {\bar B}(\F_q)}$.
\end{corollary}

\begin{proof}
The corollary is proved by inductively applying Corollary \ref{c:inj1}
to the cases where $L=K^{(1/p^m)}$ for $m=0,1,...$.
\end{proof}

Note: It can be shown that the kernel actually equals ${\widehat {\bar B}(\F_q)}$,
but we do not need it here.

\end{subsection}

\begin{subsection}{Kummer theory}\label{subsec:Kum}
Over the field ${\bar K}^{(1/p^{\infty})}$,
we are able to establish the related Kummer theory,
because we have the exact sequence of $G_K$-modules
$$0\longrightarrow A[p^m]\stackrel{j}{\longrightarrow}
A({\bar K}^{(1/p^{\infty})})\stackrel{[p^m]}{\longrightarrow} A({\bar K}^{(1/p^{\infty})})
\longrightarrow 0.
$$
Furthermore, in the exact sequence we can replace $A[p^m]$ by
${\bar A}[p^m]$ (Lemma \ref{l:lift}).
By taking the direct limit over $m$ of the induced Kummer sequence, we get
the following exact sequence:
$$
A(K^{(1/p^{\infty})})\otimes_{\Z} \Q_p/\Z_p
\rightarrowtail
\coh^1(G_K, {\bar A}[p^{\infty}])\stackrel{j_*}{\twoheadrightarrow}
\coh^1(G_K, A(K^{(1/p^{\infty})}))_p
$$
where $\coh^1(G_K, A(K^{(1/p^{\infty})}))_p$ denotes the $p$-primary part
of the group
$\coh^1(G_K, A(K^{(1/p^{\infty})}))$.
Now equations (\ref{e:pm}) and (\ref{e:dp}) together actually imply
\begin{equation}\label{e:ker}
A(K^{(1/p^{\infty})})\otimes_{\Z} \Q_p/\Z_p=A(K^{(1/p^{\infty})})_p\otimes_{\Z} \Q_p/\Z_p=0.
\end{equation}
If we set
$k_*=  j_*^{-1}\circ i_*:\coh^1(K,A)_p \longrightarrow \coh^1(G_K, {\bar A}[p^{\infty}])$,
then, by Corollary \ref{c:inj2},
\begin{equation}\label{e:kerk}
\ker(k_*)\subset {\widehat {\bar B}(\F_q)}
\end{equation}

\end{subsection}

\begin{subsection}{The proof of Theorem \ref{t:loccont}}\label{subsec:t1}
First, we restrict
the map $k_*$ to the finite $p$-group
$\coh^1(L'/K, A(L'))$. Because of the inclusion (\ref{e:kerk}), the kernel is also
contained in
${\widehat {\bar B}(\F_q)}_p$ whose order equals $|{\bar A}(\F_q)_p|$.
On the other hand, by applying the equality (\ref{e:ker}) (for the case where $K=L'$),
we can deduce that the image is contained in
$\coh^1(\Gal(L'/K),{\bar A}(\F_{L'})_p)$. According to
Corollary \ref{c:list}(e), the $\Gal(L'/K)$-module ${\bar A}(\F_{L'})_p$
equals $A(L'^{(1/p^m)})_{tor,p}$ for large enough $m$.
In view of this, Corollary \ref{c:loc} implies
$$|\coh^1(\Gal(L'/K),{\bar A}(\F_{L'})_p)|
\leq |A(K^{(1/p^m)})_{tor,p}|^d=|{\bar A}(\F_q)_p|^d.$$
And the proof is completed.

\end{subsection}
\end{section}

\end{document}